\documentclass{amsart}
\usepackage{graphicx} 
\usepackage{amsmath}
\usepackage{amssymb}
\usepackage{amsthm}
\usepackage{amsfonts}
\usepackage{biblatex}
\usepackage{breqn}
\usepackage{hyperref}

\title{Increasing Trees and the Degree-Chromatic Polynomial}
\author{Medet Jumadildayev}
\address{Nazarbayev University, Astana, Kazakhstan}
\email{medet.jumadildayev@nu.edu.kz}
\date{}

\newtheorem{thm}{Theorem}[section]
\newtheorem{dfn}[thm]{Definition}
\newtheorem{crl}[thm]{Corollary}
\newtheorem{prp}[thm]{Proposition}
\newtheorem{lm}[thm]{Lemma}

\bibliography{references} 

\begin{document}

\maketitle

\begin{abstract}
    This paper studies increasing trees on $n$ labeled vertices, in which labels increase from the root to the leaves. It is known that the number of binary increasing trees coincides with the number of alternating permutations (Euler numbers). Riordan obtained explicit formulas for the numbers of ternary and quaternary trees. This article derives a general formula for the number of $m\text{-ary}$ increasing trees for any $m$. The main result is expressed in terms of the degree-chromatic polynomial of the complete graph and Bell polynomials. It is shown how the corresponding generating function is related to the inversion problem and how combinatorial methods, including the lemma on coefficients of the multiplicative inverse function and the Lagrange inversion formula, can be used to compute the coefficients. A connection is also established between the values of the degree-chromatic polynomial at $\lambda=-1$ and the numbers of special permutations studied by Gessel.
\end{abstract}

\section{Introduction}

A labeled tree of size $n$ is a rooted tree on $n$ vertices that are labeled by $n$ distinct integers. An increasing tree is a labeled tree such that the labels along the path from the root to any leaf are increasing. An $m\text{-ary}$ increasing tree is an increasing tree such that all vertices have an out-degree at most $m$. Riordan \cite{riordan1978forests} considers the enumeration of labeled $m\text{-ary}$ rooted planar increasing trees. Similar to Riordan's notation, denote $T_n(m)$ as the number of $m\text{-ary}$ increasing trees on $n$ vertices. It is a well-known result that the number of binary increasing trees $T_n(2)$ is the Euler number $E_n$. Riordan managed to find explicit formulas for $T_n(3)$ and $T_n(4)$. 

In this paper, we obtain the formula for the number of $\text{$m$-ary}$ increasing trees for any $m$. They can be expressed using the degree-chromatic polynomials introduced in \cite{Humpert}. Section 2 outlines the necessary tools for coefficient extraction. In Section 3, we discuss the technical details of the degree-chromatic polynomial. We prove the main result in Section 4:

\begin{thm}\label{main} Let $T_n(m)$ be the number of $m\text{-ary}$ increasing trees with $n$ vertices. Let $\chi_m(G, \lambda)$ denote the degree-chromatic polynomial of $G$. We denote $B_{n, k}(x_1, x_2, \cdots, x_{n - k + 1})$ as the Bell polynomial. The number of  $m\text{-ary}$ increasing trees can be expressed in terms of $\chi_m(G, \lambda)$:
    \[
        T_n(m) = \sum_{k = 0} ^ {n} (-1) ^ k B_{n + k - 1, k} (\chi_m(K_0, -1), \chi_m(K_1, -1), \cdots, \chi_m(K_{n - 1}, -1)),
    \]
    where for any $s \geq 1$,
    \[
        \chi_m(K_s,  -1) = \sum_{k = 1} ^ s (-1) ^ k k! B_{s, k}(1_m).
    \]
\end{thm}

\section{Coefficient Extraction}

In this section, we outline the preliminary results for coefficient extraction from generating functions. For that, we define the Bell polynomial.

\begin{dfn}
    A Bell polynomial $B_{n, k}(x) = B_{n, k}(x_1, x_2, \cdots, x_{n - k + 1})$ is defined as
    $$B_{n, k}(x) = \sum_{\alpha} \frac{n!}{\alpha_1! \alpha_2! \cdots \alpha_{n - k + 1}!} \left(\frac{x_1}{1!}\right) ^ {\alpha_1} \left(\frac{x_2}{2!}\right) ^ {\alpha_2} \cdots \left(\frac{x_{n - k + 1}}{(n - k + 1)!}\right) ^ {\alpha_{n - k + 1}},$$
    where the sum is taken over all sequences $\alpha$ of non-negative integers satisfying following two conditions.
    
    \begin{enumerate}
        \item $\alpha_1 + \alpha_2 + \alpha_3 + \cdots + \alpha_{n - k + 1} = k$,
        \item $\alpha_1 +2 \alpha_2 + 3\alpha_3 + \cdots + (n - k+ 1) \alpha_{n - k + 1} = n$.
    \end{enumerate}
\end{dfn}

The coefficients of the expansion of the multiplicative inverse can be obtained using the following lemma.

\begin{lm}\label{mult_inverse} Let \(x(t) = \sum_{n\geq 0} x_n \dfrac{t ^ n}{n!}\) be an arbitrary generating function such that it has a non-zero constant term. Then, the coefficients of its multiplicative inverse \(\dfrac{1}{x(t)}\) can be computed as follows:
    \[
    \left[ \dfrac{t ^ n}{n!} \right] \dfrac{1}{x(t)} = \dfrac{1}{x_0} \sum_{k = 0} ^ {n} \dfrac{k!}{(-x_0) ^ k} B_{n, k} \left( x_1, x_2, \cdots, x_{n - k + 1}\right).
    \]
\end{lm}

\begin{proof}
    We may rewrite the expression \(\dfrac{1}{x(t)}\) as follows:

    \[
    \dfrac {1}{x(t)} = \dfrac{1}{x_0}\dfrac{1}{1 - \sum_{n \geq 1} \dfrac{x_n}{-x_0} \dfrac{t ^ n}{n!}}.
    \]

    Further, let's expand this expression using the geometric series identity \( 1 / (1 - a) = 1 + a + a^ 2 + \cdots\) to obtain:

    \begin{equation}\label{mult_aux}
        \dfrac{1}{x(t)} = \dfrac{1}{x_0}\dfrac{1}{1 - \sum_{n \geq 1} \dfrac{x_n}{-x_0} \dfrac{t ^ n}{n!}} = \dfrac{1}{x_0} \sum_{k \geq 0} \left( \sum_{n \geq 1} \dfrac{x_n}{-x_0} \frac{t ^ n}{n!} \right) ^ {k}.
    \end{equation}

    The summand can be expanded as follows:

    \[
    \left( \sum_{n \geq 0} \dfrac{x_n}{-x_0} \frac{t ^ n}{n!} \right)  ^ {k} = \dfrac{1}{(-x_0) ^ k}\sum_{n \geq 0} \dfrac{t ^ n}{n!} \sum_{\alpha} \binom{k}{\alpha_1, \alpha_2, \cdots} \binom{n}{1 ^ {\alpha_1} 2 ^ {\alpha_2} \cdots } x_1 ^ {\alpha_1} x_2 ^ {\alpha_2} \cdots,
    \]
    where $\alpha$ is a sequence which satisfies the following conditions:
    \begin{enumerate}
        \item \(\sum_{i \geq 1} \alpha_i = k\)
        \item \(\sum_{i \geq 1} i \alpha_i = n\)
    \end{enumerate}
    The expression can be rewritten using Bell Polynomials:
    \[
    \dfrac{1}{(-x_0) ^ k}\sum_{n \geq 0} \dfrac{t ^ n}{n!} \sum_{\alpha} \binom{k}{\alpha_1, \alpha_2, \cdots} \binom{n}{1 ^ {\alpha_1} 2 ^ {\alpha_2} \cdots } x_1 ^ {\alpha_1} x_2 ^ {\alpha_2}  \cdots =
    \]
    \[
    \dfrac{1}{(-x_0) ^ k}  \sum_{n \geq 0} \dfrac{t ^ n}{n!} k! B_{n, k}(x_1, x_2, \cdots, x_{n - k + 1}).
    \]
    Substituting this to equation (\ref{mult_aux}), we obtain the expansion:
    \begin{dmath*}
        \dfrac{1}{x(t)} = \dfrac{1}{x_0} \sum_{k \geq 0} \dfrac{1}{(-x_0) ^ k} \sum_{n \geq 0} \dfrac{t ^ n}{n!} k! B_{n, k}(x_1, x_2, \cdots, x_{n - k + 1}) = \sum_{n \geq 0} \dfrac{t ^ n}{n!} \left( \dfrac{1}{x_0} \sum_{k \geq 0} \dfrac{k!}{(-x_0) ^ k} B_{n, k}(x_1, x_2, \cdots, x_{n - k + 1}) \right).
    \end{dmath*}
    Then, the coefficient of $t ^ n /n!$ has the following form:
    \[
    \left[ \dfrac{t ^ n}{n!}\right] \dfrac{1}{x(t)} = \dfrac{1}{x_0} \sum_{k = 0} ^ n \dfrac{k!}{(-x_0) ^ k} B_{n, k}(x_1, x_2, \cdots, x_{n - k + 1}).
    \]
\end{proof}

Another important result is the Lagrange Inversion Formula in terms of Bell Polynomials \cite{Comtet}:

\begin{lm}\label{comtet} Let $x(t) = \sum_{n \geq 1} x_n\frac{t ^ n}{n!}$. Then, the coefficients $x_n ^ {\langle -1 \rangle}$ of its compositional inverse $x^{\langle -1 \rangle}(t)$ can be obtained using the following:

    \[
        x_n^{\langle-1 \rangle} =
        \begin{cases}
          \dfrac{1}{x_1}, & \text{if } n = 1, \\\\
          \displaystyle \sum_{k = 1} ^ {n - 1} \frac{(-1) ^ k}{x_1 ^ {n + k}} B_{n + k - 1, k}(0, x_2, x_3, \cdots, x_n), & \text{if } n > 1.
        \end{cases}
    \]
\end{lm}

\section{Degree-Chromatic polynomials}

Denote \(\chi_m(G, \lambda)\) as the number of colorings of $G$ in $\lambda$ colors such that for any vertex $v \in V(G)$, there is at most $m - 1$ vertices $u \in V(G)$ such that $u$ and $v$ are adjacent and have the same color. Humpert \cite{Humpert} showed that $\chi_m(G, \lambda)$ is a polynomial using Hopf Algebras of graphs. We call $\chi_m(G, \lambda)$ the degree-chromatic polynomial of $G$. In this section, we compute the degree-chromatic polynomial of the complete graph.

Let $K_n$ be the complete graph. Denote $1_m = (1, 1, \cdots ,1, 0, 0, \cdots)$ an infinite sequence with $m$ ones. Then, the degree-chromatic polynomial $\chi_m(K_n, \lambda)$ can be computed as follows:

\begin{lm}\label{chromatic_polynomial} Let $\chi_m(K_n, \lambda)$ denote the degree-chromatic polynomial of the complete graph $K_n$. Then, 

\[
\chi_m(K_n, \lambda) = \sum_{k = 1} ^ n B_{n, k} (1_m) (\lambda)_k,
\]
where $(\lambda)_k = \lambda(\lambda-1)\cdots(\lambda-k+1)$ is the falling factorial.
    
\end{lm}

\begin{proof}
    Let $k$ be the number of distinct colors in a coloring. By definition of degree-chromatic polynomial, a vertex can have at most $m$ adjacent vertices of the same color. Because all vertices in $K_n$ are adjacent, at most $m$ vertices can be colored with the same color. This is equivalent to partitioning the vertex set into $k$ blocks, and coloring them. The number of ways to assign each block a color is the falling factorial $(\lambda)_{k}$. Thus,
    \[
    \chi_{m} (K_n, \lambda) = \sum_{k = 1} ^ {n} B_{n, k}(1_m) (\lambda)_{k}.
    \]
\end{proof}
As a corollary, we obtain the following:

\begin{crl}\label{inverse_char} Degree chromatic polynomials of complete graphs evaluated at $\lambda=-1$ can be computed as follows:

\[
\chi_m(K_n,  -1) = \sum_{k = 1} ^ n (-1) ^ k k! B_{n, k}(1_m) 
\]
\end{crl}

Below, we provide a table of values for the degree-chromatic polynomials:

\begin{table}[h!]
\centering
\begin{tabular}{c|rrrrrrrrr}
$m \backslash n$ & 1 & 2 & 3 & 4 & 5 & 6 & 7 & 8 & 9 \\
\hline
1 &  -1 &   2 &   -6 &    24 &   -120 &     720 &   -5040 &    40320 &  -362880 \\
2 &  -1 &   1 &    0 &    -6 &     30 &     -90 &       0 &     2520 &   -22680 \\
3 &  -1 &   1 &   -1 &     2 &    -10 &      50 &    -210 &      840 &    -4200 \\
4 &  -1 &   1 &   -1 &     1 &      0 &     -10 &      70 &     -350 &     1470 \\
5 &  -1 &   1 &   -1 &     1 &     -1 &       2 &     -14 &       98 &     -546 \\
6 &  -1 &   1 &   -1 &     1 &     -1 &       1 &       0 &      -14 &      126 \\
7 &  -1 &   1 &   -1 &     1 &     -1 &       1 &      -1 &        2 &      -18 \\
8 &  -1 &   1 &   -1 &     1 &     -1 &       1 &      -1 &        1 &        0 \\
9 &  -1 &   1 &   -1 &     1 &     -1 &       1 &      -1 &        1 &       -1 \\
\end{tabular}
\caption{Table of degree-chromatic polynomials \(\chi_m(K_n, \lambda)\) evaluated at $\lambda=-1$}
\end{table}

Gessel \cite{Gessel} studied permutations of length $2m$ such that its increasing runs are congruent to $0$ or $1$ modulo $2m$. We denote the number of such permutations $a_n(m)$. We prove that, up to sign, $\chi_{2m - 1}(K_n)$ equals $a_n(m)$:

\begin{prp}\label{gessel_eq_chromatic} Let $a_n(m)$ denote the number of permutations of length $2m$ such that its increasing runs are congruent to $0$ or $1$ modulo $2m$. For $n \geq 1$, the following holds:
    \[
        \chi_{2m - 1} (K_n, -1) = (-1) ^ {n} a_n(m).
    \]
\end{prp}

\begin{proof}
    The coefficients of a partial sum of the alternating exponential series are equal to $a_n(m)$ \cite{Gessel}:

    \[
    \left[ \dfrac{t ^ n}{n!} \right] \left( 1 \big/ \sum_{k = 0} ^ {2m - 1} (-1) ^ k \dfrac{t ^ k}{k!} \right) = a_n(m).
    \]
    Using Lemma \ref{mult_inverse}, we can obtain the values for $a_n(m)$:
    \[
    a_n(m) = \sum_{k = 0} ^ n (-1) ^ k k! (-1) ^ n B_{n, k}(1_m).
    \]
    Thus, from Corollary \ref{inverse_char}, we obtain the equality:
    \[
        \chi_{2m - 1} (K_n, -1) = (-1) ^ {n} a_n(m).
    \]
\end{proof}

\section{Increasing Trees}

In this section, we outline the degree-sequence generating function for the number of increasing trees. Then, we obtain the explicit formula for its coefficients. Our main result is a corollary of that formula.

\begin{thm}\label{gf_explicit} Let the degree function $x(t) = \sum_{n \geq 0} x_n\dfrac{t ^ n}{n!}$. Denote $P(t) = P(t, x)$ the degree-sequence generating function for increasing trees. Then, $P(t)$ can be expressed as the compositional inverse: 
    \[
    P(t) = \left( \int_{0} ^ {t} \dfrac{dt}{x(t)}\right) ^ {\langle -1 \rangle}.
    \]
\end{thm}

\begin{proof}
    The following is essential to prove the theorem:
    \[
    P(P^{\langle-1\rangle}(t)) = t.
    \]
    \[
    P'(P^{\langle-1\rangle}(t)) (P ^ {\langle-1 \rangle} (t))' = 1.
    \]
    \begin{equation}\label{inverse}
        (P^{\langle-1 \rangle}(t))' = \dfrac{1}{P'(P^{\langle-1 \rangle}(t))}.
    \end{equation}

    The generating function $P(t)$ can be given implicitly by constructing forests consisting of $k$ trees, and attaching their roots to a new root with weight $x_k$. Integration increases the number of nodes, thus accounting for the addition of the root:

    \[
    P(t) =  \sum_{k \geq 0} \int_{0} ^ t \left( x_k\dfrac{P(t) ^ k}{k!} \right) dt = \int_0 ^ t x(P(t)) dt.
    \]
    To obtain $P(t)$ we solve the differential equation:
    \[
    P(t) = \int_0 ^ t x(P(t)) dt.
    \]
    Differentiate both sides with respect to $t$:
    \[
    P'(t) = x(P(t)).
    \]
    Equivalently,
    \[
    P'(P^{\langle-1\rangle}(t)) = x(t).
    \]
    Applying Equation (\ref{inverse}), we get the following:
    \[
    \frac{1}{x(t)} = (P ^ {\langle-1\rangle} (t))'.
    \]
    Hence,
    \[
    \int_{0} ^ t \dfrac{dt}{x(t)} = P ^ {\langle -1 \rangle}(t).
    \]
    Finally,
    \[
    P(t) = \left(\int_{0} ^ t \dfrac{dt}{x(t)}\right) ^ {\langle -1 \rangle}.
    \]
\end{proof}

Theorem \ref{gf_explicit} is a well-known result \cite{Flajolet}. Now, we calculate the coefficients of $P(t)$ using Lemma \ref{mult_inverse} and Lemma \ref{comtet}:

\begin{thm}\label{closed_form} Let $p_n$ denote the coefficient \( \left[\dfrac{t ^ n}{n!} \right]P(t)\). Then, $p_n$ can be expressed as follows:

\[
p_n = x_0 ^ n \sum_{k = 0} ^ {n} (-x_0) ^ k B_{n + k - 1, k} \left( 0, b_1, b_2, \cdots, b_{n - 1}\right),
\]
where $b_s$ are auxiliary variables defined as follows:
\[
b_s = \dfrac{1}{x_0}\sum_{k = 0} ^ s \frac{k!}{(-x_0) ^ {k}} B_{s, k}(x_1, x_2, \cdots, x_{s - k + 1}).
\]
\end{thm}

\begin{proof}
    By Lemma \ref{mult_inverse}, the coefficients of $\frac{1}{x(t)}$, which we denote as $b_s$, are expressed as follows:
    \[
    b_s = \left[ \dfrac{t ^ s}{s!}\right] \dfrac{1}{x(t)} = \dfrac{1}{x_0} \sum_{k = 0} ^ s \dfrac{k!}{(-x_0) ^ k} B_{s, k}(x_1, x_2, \cdots, x_{n - k + 1}).
    \]
    In particular,
    \[
    b_0 = \dfrac{1}{x_0}.
    \]
    Then, for $s \geq 1$, the coefficients of integration \( \int_0 ^ t \frac{dt}{x(t)}\) are the same as shifting the sequence to the right:
    \begin{equation} \label{integral1}
        \left[ \dfrac{t ^ s}{s!}\right] \int_0 ^ t \frac{dt}{x(t)} = b_{s - 1}.
    \end{equation}
    Note that the expression \ref{integral1} has no constant term, which means that we can apply Lemma \ref{comtet}. Thus, for $n > 1$, we obtain the coefficients of the compositional inverse:
    \begin{dmath*}
        p_n = \left[ \dfrac{t ^ n}{n!}\right] p(t) = \left[ \dfrac{t ^ n}{n!}\right] \left( \int_0 ^ t \frac{dt}{x(t)} \right) ^ {-1} = \sum_{k = 1} ^ {n - 1} \dfrac{(-1) ^ k}{a_0 ^ {n + k}} B_{n + k - 1, k} \left( 0, b_1, b_2, \cdots, b_{n - 1}\right) = \sum_{k = 0} ^ {n - 1} (-1) ^ k x_0 ^ {n + k} B_{n + k - 1, k} \left( 0, b_1, b_2, \cdots, b_{n - 1}\right).
    \end{dmath*}
\end{proof}

Our main result is the corollary of Theorem \ref{closed_form}:

\begin{proof}[Proof of Theorem \ref{main}]
    Setting $x_0 = x_1 = \cdots = x_m = 1$ and $x_{m + 1} = x_{m + 2} = \cdots = 0$ gives the number of $m\text{-ary}$ increasing trees:

    \[
    T_n(m) = \sum_{k = 0} ^ {n} (-1) ^ k B_{n + k - 1, k}(0, b_1, b_2, \cdots, b_{n - 1}),
    \]
    where $b_s$ is defined as follows:
    \[
    b_s = \sum_{k = 0} ^ s (-1) ^ k k! B_{s, k}(1_m)
    \]

    Note that by Corollary \ref{inverse_char}, we obtain that $b_s = \chi_m(K_s, -1)$. Also, $\chi_m(K_0, -1) = 0$. As a result, we obtain the result:

    \[
    T_n(m) = \sum_{k = 0} ^ {n} (-1) ^ kB_{n + k - 1, k}(\chi_m(K_0, -1), \chi_m(K_1, -1), \cdots, \chi_m(K_{n - 1}, -1)).
    \]
\end{proof}

We provide the table with entries $T_n(m)$, which is also given in \cite{riordan1978forests}. Note that $T_n(2)$ gives a sequence for Euler numbers $E_n$. It is obvious that the values $T_n(m)$ for $n \leq m$ are equal to $(n - 1)!$, so we left their respective entries empty.

\begin{table}[h!]
\centering
\begin{tabular}{c|rrrrrrrrrr}
$m \backslash n$ & 2 & 3 & 4 & 5 & 6 & 7 & 8 & 9 & 10 \\
\hline
1  & 1 & 1 & 1 & 1 & 1 & 1 & 1 & 1 & 1 \\
2   &  & 2 & 5 & 16 & 61 & 272 & 1385 & 7936 & 50521 \\
3   &   &  & 6 & 23 & 108 & 601 & 3863 & 28159 & 229524 \\
4   &   &   &  & 24 & 119 & 703 & 4819 & 37596 & 328871 \\
5   &   &   &   &  & 120 & 719 & 5017 & 39938 & 357100 \\
6   &   &   &   &   &  & 720 & 5039 & 40290 & 362258 \\
7   &   &   &   &   &     &  & 5040 & 40319 & 362842 \\
8   &   &   &   &   &     &     &  & 40320 & 362879 \\
9   &   &   &   &   &     &     &      &  & 362880 \\
\end{tabular}
\caption{Table for the number of $m\text{-ary}$ labeled increasing trees $T_n(m)$}
\end{table}

\section{Acknowledgements}

This research was funded by the  Science Committee of the Ministry of Science and Higher Education of the Republic of Kazakhstan (Grant No. BR 28713025). 

\printbibliography

\end{document}